\documentclass[12pt]{article} 
 \usepackage{enumerate}     
 \usepackage{longtable} 
 \usepackage{mathrsfs} 
 \usepackage{amsmath} 
 \usepackage{graphicx}     
  \usepackage{tikz}
\usetikzlibrary{positioning, calc, shapes.geometric, shapes, shapes.multipart, arrows.meta, arrows, decorations.markings, external, trees}
 \usepackage{float}
 \usepackage{indentfirst} 
 \usepackage{makeidx}\usepackage{latexsym}\usepackage{amssymb}\usepackage{color}
\usepackage{amsthm}
\usepackage{titlesec}
\usepackage{geometry}
\usepackage{setspace} 
\usepackage{caption} 
\usepackage{authblk} 
\usepackage{url}
\usepackage{titlesec}
\usepackage{hyperref}
\usepackage[american]{babel}

\allowdisplaybreaks

\newtheorem{theorem}{Theorem}

\newtheorem{example}{Example}[section]

\title{\textbf{Bounding the probability of causality under ordinal outcomes}}

\author[1]{Hanmei Sun}
\author[1]{Chengfeng Shi}
\author[1*]{Qiang Zhao}

\affil[1]{School of Mathematics and Statistics, Shandong Normal University, Jinan, Shandong, 250014, China.}

\date{}

\begin{document}

\maketitle

\let\thefootnote\relax\footnotetext{
    *Corresponding author. E-mail(s): qzhao@sdnu.edu.cn; \\
    Contributing authors: hmsun@sdnu.edu.cn; 2022020521@stu.sdnu.edu.cn; \\
}
\begin{abstract}
The probability of causation (PC) is often used in liability assessments. In a legal context, for example, where a patient suffered the side effect after taking a medication and sued the pharmaceutical company as a result, the value of the PC can help assess the likelihood that the side effect was caused by the medication, in other words, how likely it is that the patient will win the case. Beyond the issue of legal disputes, the PC plays an equally large role when one wants to go about explaining causal relationships between events that have already occurred in other areas. This article begins by reviewing the definitions and bounds of the probability of causality for binary outcomes, then generalizes them to ordinal outcomes. It demonstrates that incorporating additional mediator variable information in a complete mediation analysis provides a more refined bound compared to the simpler scenario where only exposure and outcome variables are considered.
\end{abstract}

\section{Introduction}

Statisticians are often interested in exploring causal relationships between two events,  typically involving two types of questions: the effects of causes (EOC) and the causes of observed effects (COE). The former focuses on predicting and inferring the effects of interventions in groups, which can be addressed through simple experimental designs or statistical methods in observational studies. The latter aims to determine whether a specific cause led to an observed outcome in an individual case, in other words, whether there is a causal link between a particular exposure and the resulting outcome. In particular, epidemiologists endeavour to investigate whether a disease was caused by a certain exposure (e.g., taking medication). Consider the case where Ann experienced a headache, took a medication, and her headache disappeared shortly afterward. The question arises: was the disappearance of Ann’s headache caused by the medication? The study of such questions will inevitably lead us into counterfactual thinking. Specifically, given that Ann took the medication and her headache resolved, we need to consider: would her headache have disappeared if she had not taken the medication? If Ann’s headache persists without the medication, it would be reasonable to attribute her headache's disappearance to the medication. However, since we can never observe the outcome when Ann does not take the medication, establishing causality between two events remains challenging. Consequently, we cannot directly explain the causal relationship between events observed in individual cases based solely on counterfactual reasoning. A natural idea is to find a quantity that measures the probability that an outcome is attributable to a particular cause.

Generally speaking, EOC problems are easier to study than COE problems, which means that the former are more likely to be successful in empirical studies. However, most of the existing research work on causality has considered how to deal with COE problems, which is also the focus of this paper. Dawid (2015) described and contrasted how EOC problems can be tackled under decision theory, structural equations, and the potential outcome framework (Neyman 1923; Rubin 1974), and showed that the potential outcome framework is not necessary to deal with EOC problems. In contrast, the counterfactual theoretical logic is necessary for dealing with COE problems (Dawid and Musio 2022). Within the counterfactual framework, we can define the probability of causality, which allows us to effectively handle COE problems of interest.

The probability of causality has broad applications across various fields, including legal disputes, health sciences, and market analyses. It helps assess the likelihood that one event is attributable to another. Pearl (1999) presented definitions of three aspects of probability of causality: the probability of necessity (PN), the probability of sufficiency (PS), and the probability of sufficient necessity (PNS). Tian and Pearl (2000) demonstrated how both experimental and non-experimental data can be used to identify the probability of causality under assumptions such as exogeneity and monotonicity. Building on this, Kuroki and Cai (2011) incorporated additional covariate information to obtain more precise bounds on the probability of causality. Dawid et al. (2016) defined the probability of causality based on the probability of necessity defined by Pearl and stated that: in general, we can’t give a precise estimation of the PC, but we can give an upper and lower bound of the PC under certain assumptions. When the value of PC exceeds 0.5, a causal link between events can be established on a ``balance of probabilities" basis. Lu et al. (2023) further considered the attribution problem in a multivariate context, proposing the posterior total causal effect and posterior direct causal effect (Li, Lu et al. 2024), and proving the identifiability of these causal effects under the assumptions of no confounding and monotonicity. Most existing studies focus on binary exposure and outcome variables, Li and Pearl (2024) showed how observational and experimental data can be used to bound the probability of causality when both variables are non-binary, demonstrating the application of theoretical limits of PC in various contexts. Many practical issues involve non-binary exposures or outcomes, highlighting the need for further refinement and development of the theoretical study of probability of causality in these areas.

In this article, we focus on considering bounds on the probability of causality under ordinal outcomes in complete mediation analysis. Section 2 gives the notation used to define the probability of causality and the basic assumptions needed for the subsequent theoretical analysis, and gives the bounds on the probability of causality under binary outcomes. We further consider the form of the probability of causality under ordinal outcomes in Section 3, and derive bounds on the PC in two different cases using similar principles as in Section 2. In Section 4, we analyze the bounds of probability of causality under three-valued outcomes using numerical simulations. Section 5 summarizes the research in this paper and provides an outlook for subsequent research.

\section{Probability of causality under binary outcomes}
This section reviews the definition of the probability of causality for binary exposure and outcome variables, introduces the notation and assumptions used, and presents the upper and lower bounds for two distinct scenarios.
\subsection{Simple scenario}

Firstly, consider the bounds on the probability of causality when only information about the exposure and outcome variables is available. In Ann's case, assume that there is a binary exposure variable $D\in \{0,1\}$, where $D=1$ if Ann took the medication and $D=0$ otherwise. The outcome variable $Y$ is also binary, with $Y=1$ indicating that the headache disappeared and $Y=0$ otherwise. Define $\mathbf{Y}:=(Y(0),Y(1))$, where $Y(d)$ represents the potential outcome of $Y$ if $D$ is set to the value $d$ by external intervention. Dawid, Murtas and Musio (2016) used the potential outcome framework to define the probability of causality in Ann's case:
\begin{equation}
 P{{C}_{A}}={{P}_{A}}\left( {{Y}_{A}}(0)=0|{{D}_{A}}=1,{{Y}_{A}}(1)=1 \right),         
 \end{equation}
 where ${{P}_{A}}$ denotes the probability distribution of the corresponding attribute of Ann.

Next, we need to introduce two basic assumptions used to infer causality in Ann's example.

\noindent$\mathbf{Assumption}$ $\mathbf{1}$ (D-Y no confounding). $D$ is independent of $\mathbf{Y}$.

\noindent$\mathbf{Assumption}$ $\mathbf{2}$ (Exchangeability). Ann is exchangeable with the individuals in the study population.

Assumption 1 implies that there are no confounders between $D$ and $Y$. Under Assumption 2, data from a population exposed and unexposed to the same drug as Ann can be used to identify the probability distribution of Ann's corresponding attributes.

Under Assumptions 1 and 2, Eq. (1) is equivalent to
\begin{equation}
 P{{C}_{A}}=\Pr (Y(0)=0|Y(1)=1).
 \end{equation}

Eq. (2) reflects the proportion of individuals whose headaches do not resolve without medication among those whose headaches would resolve if they took medication. In other words, it indicates the probability that the disappearance of a headache can be attributed to the medication. The denominator of Eq. (2) can be directly estimated from the data, while the numerator represents the probability of a pair of counterfactual values, we cannot determine this equation directly from the data, but can only give the corresponding bounds. As shown in Dawid et al. (2016), the bounds for ${{PC}_{A}}$ are as follows:

\begin{equation}
1-\frac{1}{RR}\le P{{C}_{A}}\le \frac{\Pr (Y=0|D=0)}{\Pr (Y=1|D=1)},
 \end{equation}
where 
$$RR=\frac{\Pr (Y=1|D=1)}{\Pr (Y=1|D=0)}$$
denotes the causal risk ratio. When $RR>2$, the probability of causality ${{PC}_{A}}>0.5$, so that the causal relationship between events can be determined on a ``balance of probabilities" basis.
 
\subsection{Complete mediator}

This subsection explores the bounds on the probability of causality under the causal path depicted in Figure 1. For simplicity, let $M$ is a binary mediator variable. We aim to determine the bounds on the probability of causality given additional information about the mediator and assuming that exposure has no direct effect on the outcome. Let $M(d)$ represents the potential outcome of the mediator when $D$ is set to the value $d$. Similarly, let ${{Y}^{*}}(m)$ denotes the counterfactual value of $Y$ when $M$ is set to $m$ and $D$ is set to $d$. It is straightforward to see that ${{Y}^{*}}(M(d))=Y(d)$.

\begin{figure}[H]                          
  \centering                              
  \includegraphics[scale=1.4]{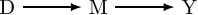}       
  \caption{Causal diagrams in which the effect of $D$ on $Y$ is completely mediated by $M$} \label{fig1}
\end{figure}

Define $\mathbf{M}:=(M(0),M(1))$, $\mathbf{{Y}^{*}}:=({{Y}^{*}}(0),{{Y}^{*}}(1))$. Referring to the way it was defined in Dawid et al (2016), we use the following symbols to denote potential response pairs as well as the conditional distributions of $M$ and $Y$ :
\begin{gather*}
{{m}_{pq}}:=\Pr (M(0)=p,M(1)=q)  \\
{{y}_{pq}}:=\Pr (Y(0)=p,Y(1)=q)  \\
{{y}^{*}}_{pq}:=\Pr ({{Y}^{*}}(0)=p,{{Y}^{*}}(1)=q) \\
{{m}_{p+}}:=\Pr (M=p|D=0)  \\
{{m}_{+q}}:=\Pr (M=q|D=1)  \\
{{y}_{p+}}:=\Pr (Y=p|D=0) \\
{{y}_{+q}}:=\Pr (Y=q|D=1)  \\
{{y}^{*}}_{p+}:=\Pr (Y=p|M=0)  \\
{{y}^{*}}_{+q}:=\Pr (Y=q|M=1). \nonumber \\
\end{gather*}
In addition to the exchangeability hypothesis, this subsection also needs to introduce the no-confounding hypothesis between D, M, and Y:

\noindent$\mathbf{Assumption}$ $\mathbf{3}$ (D-M-Y no confounding). $D$, $\mathbf{M}$ and $\mathbf{{Y}^{*}}$ are mutually independent.

Under Assumptions 2 and 3, Dawid et al. (2016) illustrated that, given the causal mechanism depicted in Figure 1, the lower bound on the probability of causality is identical to the lower bound in Eq. (3), while the upper bound is improved compared to the case without the mediator variable, which we denote here as
\begin{equation}
\frac{{{\alpha }_{1}}{{\alpha }_{2}}+({{\alpha }_{1}}+{{\beta }_{1}})({{\alpha }_{2}}+{{\beta }_{2}})}{{{y}_{+1}}},
 \end{equation}
where ${{\alpha }_{1}}=\min ({{m}_{0+}},{{m}_{+1}})$, ${{\alpha }_{2}}=\min ({{y}^{*}}_{0+},{{y}^{*}}_{+1})$, ${{\beta }_{1}}={{m}_{+0}}-{{m}_{0+}}$, ${{\beta }_{2}}={{y}^{*}}_{+0}-{{y}^{*}}_{0+}$.

\begin{example}
Suppose the following probabilities can be obtained from the data:
\begin{align*}
&\Pr \left( M=1|D=1 \right)=0.7    &\Pr \left( M=1|D=0 \right)=0.15 \\
&\Pr \left( Y=1|M=1 \right)=0.95   &\Pr \left( Y=1|M=0 \right)=0.2.
 \end{align*}
According to Eqs. (3) and (4), it is calculated that $0.57\le P{{C}_{A}}\le 0.78$. A wider bound is obtained if the information on the mediator variable is not utilized: $0.57\le P{{C}_{A}}\le 0.95$.
\end{example}

\section{Probability of causality under ordinal outcomes}

It is known that certain outcomes of interest in fields such as education, epidemiology and the social sciences are ordinal. Therefore, it makes sense to extend the relevant conclusions from attribution studies to the case of ordinal outcomes. This section builds on Section 2 by inferring bounds on the probability of causality in the context of ordinal outcomes.

\subsection{Simple scenario}

The first consideration is how to define and derive bounds on the probability of causality when generalizing the binary outcomes from Section 2.1 to ordinal outcomes. We continue with the example of Ann's headache, but now assume that $Y \in \{0, 1, \ldots, T\}$ represents an ordinal outcome. For instance, $Y$ denotes the degree of headache, with higher values indicating less severe symptoms. Suppose Ann experienced a headache with a level of $Y = t$ and took the medication, resulting in an observed outcome of $Y > t$. The question we are interested in is whether or not Ann's headache reduction was caused by taking the medication. To address this, we need to evaluate the probability that Ann's observed outcome would be $Y \le t$ had she not taken the medication. In other words, how likely it is that Ann's headache level would not have decreased without taking the medication. Analogous to the definition of the probability of causality for binary outcomes, the probability of causality is defined here as:

\begin{equation}
P{{C}_{A}}={{P}_{A}}({{Y}_{A}}(0)\le t|{{D}_{A}}=1,{{Y}_{A}}(1)>t).
\end{equation}
It is easy to see that the larger the value of ${{PC}_{A}}$ the more it indicates that Ann's medication reduces her headaches.

When Assumptions 1 and 2 hold, an equivalent form of Eq. (5) can be obtained
\begin{equation}
P{{C}_{A}}=\Pr (Y(0)\le t|Y(1)>t),
\end{equation}
where Pr denotes the probability in the entire population. Further simplifying Eq. (6) using the notation defined in the previous section, we have
\begin{equation}
P{{C}_{A}}=\frac{\sum\nolimits_{k>t}{\sum\nolimits_{l\le t}^{{}}{{{y}_{lk}}}}}{\sum\nolimits_{k>t}{{{y}_{+k}}}}.
\end{equation}
The denominator of Eq. (7) can be estimated directly from the data, similarly to the case of binary values, we can only obtain the upper and lower bounds of the numerator, and then we arrive at the following result.
\begin{theorem}
Considering binary exposure and ordinal outcome variables, the lower and upper bounds on the probability of causality are denoted by $PCL$ and $PCU$, respectively. Under Assumptions 1 and 2, the bounds on ${{PC}_{A}}$ defined by Eq. (5) are as follows
 \end{theorem}
\begin{equation} 
PCL= \max \left\{ 
   0,  \frac{1}{\sum\nolimits_{k>t}{{{y}_{+k}}}} \times \left(\sum\nolimits_{k>t}{{{y}_{+k}}}-\sum\nolimits_{l>t}{{{y}_{l+}}} \right)\\ 
 \right\},
\end{equation}
\begin{equation}
PCU=\min \left\{1,\frac{\sum\nolimits_{l\le t}{{{y}_{l+}}}}{\sum\nolimits_{k>t}{{{y}_{+k}}}} \right\}.
\end{equation}
\begin{proof}
Note that
\[\left\{ \begin{matrix}
   \sum\nolimits_{k>t}{\sum\nolimits_{l\le t}{{{y}_{lk}}+\sum\nolimits_{k>t}{\sum\nolimits_{l>t}{{{y}_{lk}}=\sum\nolimits_{k>t}{{{y}_{+k}}}}}}}  \\
   \sum\nolimits_{k>t}{\sum\nolimits_{l\le t}{{{y}_{kl}}+\sum\nolimits_{k>t}{\sum\nolimits_{l>t}{{{y}_{kl}}=\sum\nolimits_{k>t}{{{y}_{k+}}}}}}}.  \\
\end{matrix} \right.\]
Thus, the numerator of Eq. (7) can be written as follows
\begin{equation}
\begin{split}
\sum\nolimits_{k>t}{\sum\nolimits_{l\le t}^{{}}{{{y}_{lk}}}}&=\sum\nolimits_{k>t}{\sum\nolimits_{l\le t}^{{}}{{{y}_{kl}}+(\sum\nolimits_{k>t}{{{y}_{+k}}-\sum\nolimits_{l>t}{{{y}_{l+}}}}}})\\
&\ge \sum\nolimits_{k>t}{{{y}_{+k}}-\sum\nolimits_{l>t}{{{y}_{l+}}}}.\nonumber \\
\end{split}
\end{equation}

\noindent From the above, it can be concluded that the Eq. (8) holds.

Next consider the upper bound on ${{PC}_{A}}$. Notice that
\[\left\{ \begin{matrix}
   \sum\nolimits_{k>t}{\sum\nolimits_{l\le t}{{{y}_{lk}}+\sum\nolimits_{k>t}{\sum\nolimits_{l>t}{{{y}_{lk}}=\sum\nolimits_{k>t}{{{y}_{+k}}}}}}}  \\
   \sum\nolimits_{k>t}{\sum\nolimits_{l\le t}{{{y}_{lk}}+\sum\nolimits_{k\le t}{\sum\nolimits_{l\le t}{{{y}_{lk}}=\sum\nolimits_{l\le t}{{{y}_{l+}}}}}}}. \\
\end{matrix} \right.\]
Thus, we have
\[\sum\nolimits_{k>t}{\sum\nolimits_{l\le t}^{{}}{{{y}_{lk}}}}\le \min \{\sum\nolimits_{k>t}{{{y}_{+k}},}\sum\nolimits_{l\le t}{{{y}_{l+}}}\}.\]
Combine with Eq. (7), Eq. (9) holds.

\end{proof}

\subsection{Complete mediator}
In this subsection, we will consider bounds on the probability of causality in the case where mediator variable $M$ completely mediates the effect of $D$ on $Y$ and $Y$ is an ordinal outcome variable, as shown in Figure 1.

When Assumptions 2 and 3 hold, we can also obtain that Eq. (7) holds. At this point the numerator of Eq. (7) can be expressed as
\begin{equation}
\sum\nolimits_{k>t}{\sum\nolimits_{l\le t}{{{y}^{*}}_{lk}{{m}_{01}}+}{{y}^{*}}_{kl}{{m}_{10}}}.
\end{equation}
\noindent Eq. (10) further considers the information of the mediator variable, and is not identifiable. The following theorem provides bounds on the probability of causality of ordinal outcomes in complete mediation analysis based on the upper and lower bounds of Eq. (10).
\begin{theorem}
Considering binary exposure, binary mediator, and ordinal outcome variables, where exposure has no direct effect on the outcome variable, and denoting the lower and upper bounds on the probability of causality by ${{PCL}_{M}}$ and ${{PCU}_{M}}$, respectively, the bounds on ${{PC}_{A}}$ defined by Eq. (5) under Assumptions 2 and 3 are as follows:
 \end{theorem}
\begin{equation}
   PC{{L}_{M}}=  \max \left\{ \begin{split}0,\frac{1}{\sum\nolimits_{k>t}{{{y}_{+k}}}}\times cB
\end{split} \right\}, \\ 
\end{equation}

\begin{equation}
PC{{U}_{M}}=\frac{bC+(b+c)(C+B)}{\sum\nolimits_{k>t}{{{y}_{+k}}}}.
\end{equation}
Where $B=\sum\nolimits_{l>t}{{{y}^{*}}_{l+}}-\sum\nolimits_{k>t}{{{y}^{*}}_{+k}}$, $c={{m}_{+0}}-{{m}_{0+}}$,  $C=\min (\sum\nolimits_{k>t}{{{y}^{*}}_{+k}},\sum\nolimits_{l\le t}{{{y}^{*}}_{l+}})$, $b=\min \{{{m}_{0+}},{{m}_{+1}}\}$.

\noindent It is easy to see that the upper bound on the probability of causality given by Eq. (12) does not exceed 1.

\begin{proof}
Note that
\[\left\{ \begin{matrix}
   {{m}_{00}}+{{m}_{01}}={{m}_{0+}}  \\
   {{m}_{00}}+{{m}_{10}}={{m}_{+0}}  \\
\end{matrix} \right.,\] and \[\left\{ \begin{matrix}
   \sum\nolimits_{k>t}{\sum\nolimits_{l\le t}{{{y}^{*}}_{lk}+\sum\nolimits_{k>t}{\sum\nolimits_{l>t}{{{y}^{*}}_{lk}=\sum\nolimits_{k>t}{{{y}^{*}}_{+k}}}}}}  \\
   \sum\nolimits_{k>t}{\sum\nolimits_{l\le t}{{{y}^{*}}_{kl}+\sum\nolimits_{k>t}{\sum\nolimits_{l>t}{{{y}^{*}}_{kl}=\sum\nolimits_{k>t}{{{y}^{*}}_{k+}}}}}} \\
\end{matrix} \right..\]
Thus, Eq. (10) can be written as follows
\begin{equation}
\begin{split}
&\sum\nolimits_{k>t}{\sum\nolimits_{l\le t}^{{}}{{{y}^{*}}_{lk}{{m}_{01}}+}{{y}^{*}}_{kl}{{m}_{10}}} \nonumber \\
=&{{m}_{01}}\sum\nolimits_{k>t}{\sum\nolimits_{l\le t}{{{y}^{*}}_{lk}+({{m}_{01}}+{{m}_{+0}}-{{m}_{0+}})(\sum\nolimits_{k>t}{\sum\nolimits_{l\le t}{{{y}^{*}}_{lk}}+}\sum\nolimits_{l>t}{{{y}^{*}}_{l+}-\sum\nolimits_{k>t}{{{y}^{*}}_{+k}}})}}. \nonumber
\end{split}
\end{equation}
Let $B=\sum\nolimits_{l>t}{{{y}^{*}}_{l+}}-\sum\nolimits_{k>t}{{{y}^{*}}_{+k}}$, $c={{m}_{+0}}-{{m}_{0+}}$, we have
\begin{equation}
\begin{split}
    &\sum\nolimits_{k>t}{\sum\nolimits_{l\le t}^{{}}{{{y}^{*}}_{lk}{{m}_{01}}+}{{y}^{*}}_{kl}{{m}_{10}}} \nonumber \\
\ge & \max \{0,-c\}\times \max \{0,-B\}+(\max \{0,-c\}+c)\times (\max \{0,-B\}+B) \nonumber \\ 
  = & \max \{0,cB\}. \nonumber
\end{split}
\end{equation}
From the above, it can be concluded that the Eq. (11) holds.

Next consider the upper bound on ${{PC}_{A}}$. Notice that
\[\left\{ \begin{matrix}
   {{m}_{00}}+{{m}_{01}}={{m}_{0+}}  \\
   {{m}_{11}}+{{m}_{01}}={{m}_{+1}}  \\
\end{matrix} \right.,\] and \[\left\{ \begin{matrix}
   \sum\nolimits_{k>t}{\sum\nolimits_{l\le t}{{{y}^{*}}_{lk}+\sum\nolimits_{k>t}{\sum\nolimits_{l>t}{{{y}^{*}}_{lk}=\sum\nolimits_{k>t}{{{y}^{*}}_{+k}}}}}}  \\
   \sum\nolimits_{k>t}{\sum\nolimits_{l\le t}{{{y}^{*}}_{lk}+\sum\nolimits_{k\le t}{\sum\nolimits_{l\le t}{{{y}^{*}}_{lk}=\sum\nolimits_{l\le t}{{{y}^{*}}_{l+}}}}}}  \\
\end{matrix} \right..\]
Let $b=\min \{{{m}_{0+}},{{m}_{+1}}\}$, $C=\min (\sum\nolimits_{k>t}{{{y}^{*}}_{+k}},\sum\nolimits_{l\le t}{{{y}^{*}}_{l+}})$, we have
\begin{equation}
\begin{split}
    & \sum\nolimits_{k>t}{\sum\nolimits_{l\le t}^{{}}{{{y}^{*}}_{lk}{{m}_{01}}+}{{y}^{*}}_{kl}{{m}_{10}}} \nonumber  \\ 
  = & {{m}_{01}}\sum\nolimits_{k>t}{\sum\nolimits_{l\le t}{{{y}^{*}}_{lk}+({{m}_{01}}+c)(\sum\nolimits_{k>t}{\sum\nolimits_{l\le t}{{{y}^{*}}_{lk}}+}B)}} \nonumber  \\ 
\le &  bC+(b+c)(C+B). \nonumber
\end{split}
\end{equation}
Combine with Eqs. (7) and (10), Eq. (12) holds.
\end{proof}
Next we will compare the upper and lower bounds on the probability of causality given by the two theorems above, and similarly, the consideration of extra mediator variable information under ordinal outcomes yields a narrower bound:
\begin{theorem}
 In the bounds proposed by Theorem 2, the upper bound is improved, while the lower bound remains the same as in Theorem 1. Specifically, we have $PCL= PC{{L}_{M}}, PC{{U}_{M}}\le PCU$.
\end{theorem}
\begin{proof}

Firstly, considering the numerator of the fraction in Eq. (11), we have
\begin{equation}
\begin{split}
    & ({{m}_{+0}}-{{m}_{0+}})(\sum\nolimits_{l>t}{{{y}^{*}}_{l+}}-\sum\nolimits_{k>t}{{{y}^{*}}_{+k}}) \\ 
  = & ({{m}_{+0}}-{{m}_{0+}})\sum\nolimits_{l>t}{{{y}^{*}}_{l+}}+({{m}_{+1}}-{{m}_{1+}})\sum\nolimits_{k>t}{{{y}^{*}}_{+k}} \\ 
  = & {{m}_{+0}}\sum\nolimits_{l>t}{{{y}^{*}}_{l+}}+{{m}_{+1}}\sum\nolimits_{k>t}{{{y}^{*}}_{+k}}-({{m}_{0+}}\sum\nolimits_{l>t}{{{y}^{*}}_{l+}}+{{m}_{1+}}\sum\nolimits_{k>t}{{{y}^{*}}_{+k}}). \nonumber \\ 
\end{split}
\end{equation}
In addition, it is easy to obtain the following equations:
\begin{equation}
\left\{\begin{split}
  & {{y}_{+k}}={{m}_{+0}}{{y}^{*}}_{k+}+{{m}_{+1}}{{y}^{*}}_{+k} \\ 
 & {{y}_{l+}}={{m}_{0+}}{{y}^{*}}_{l+}+{{m}_{1+}}{{y}^{*}}_{+l} \nonumber \\ 
\end{split} \right..
\end{equation}
Thus,
   \[  ({{m}_{+0}}-{{m}_{0+}})(\sum\nolimits_{l>t}{{{y}^{*}}_{l+}}-\sum\nolimits_{k>t}{{{y}^{*}}_{+k}})
  = \sum\nolimits_{k>t}{{{y}_{+k}}}-\sum\nolimits_{l>t}{{{y}_{l+}}}. \]
Combine with Eq. (8), it can be concluded that $PCL= PC{{L}_{M}}$.

Next prove that $PC{{U}_{M}}\le PCU$:
\begin{equation}
\begin{split}
  & bC+(b+c)(C+B) \\
  = & \min \{{{m}_{0+}},{{m}_{+1}}\}\times \min \{\sum\nolimits_{k>t}{{{y}^{*}}_{+k},\sum\nolimits_{l\le t}{{{y}^{*}}_{l+}}}\}+\min \{{{m}_{+0}},{{m}_{1+}}\}
  \times \min \{\sum\nolimits_{k\le t}{{{y}^{*}}_{+k},\sum\nolimits_{l>t}{{{y}^{*}}_{l+}}}\} \\ 
 \le &  {{m}_{0+}}\sum\nolimits_{l\le t}{{{y}^{*}}_{l+}}+{{m}_{1+}}\sum\nolimits_{k\le t}{{{y}^{*}}_{+k}} \\ 
  = & \sum\nolimits_{l\le t}{{{y}_{l+}}}. \nonumber \\ 
\end{split}
\end{equation}
Combine with Eq. (9), it can be concluded that $PC{{U}_{M}}\le PCU$.
\end{proof}

\begin{example}
Consider the case where Y takes on three values, suppose the following probabilities can be obtained from the data:
\begin{align*}
 &\Pr \left( M=1|D=1 \right)=0.95       &\Pr &\left( Y=1|M=0 \right)=0.1 \\
 &\Pr \left( M=1|D=0 \right)=0.15       &\Pr &\left( Y=0|M=1 \right)=0.25 \\
 &\Pr \left( Y=1|M=1 \right)=0.05       &\Pr &\left( Y=0|M=0 \right)=0.8.
\end{align*}
According to Eqs. (11) and (12), it is calculated that $0.71\le P{{C}_{A}}\le 0.89$. A wider bound is obtained if the information on the mediator variable is not utilized: $0.71\le P{{C}_{A}}\le 1$.
\end{example}
\section{Simulation studies}
In this section, we will assess the quality of the bounds on the probability of causality given in Theorem 2 and compare the bounds in Theorems 1 and 2 by analyzing the midpoints of their respective upper and lower bounds.

Consider the case where T = 2 and t = 1(i.e., Y has three values), in this case $P{{C}_{A}}$ can be expressed as
\[\frac{({{y}^{*}}_{02}+{{y}^{*}}_{12}){{m}_{01}}+({{y}^{*}}_{20}+{{y}^{*}}_{21}){{m}_{10}}}{{{y}^{*}}_{+2}{{m}_{+1}}+{{y}^{*}}_{2+}{{m}_{+0}}}.\]
We first randomly generated 100 samples of  $P{{C}_{A}}$, and then, we generated sample distributions for each sample that matched the $P{{C}_{A}}$. Following this, we draw the graph of the real value of the probability of causality along with its upper and lower bounds as determined by Theorem 2. The results are shown in Figure 2.

\begin{figure}[H]                           
  \centering                               
  \includegraphics[scale=0.7]{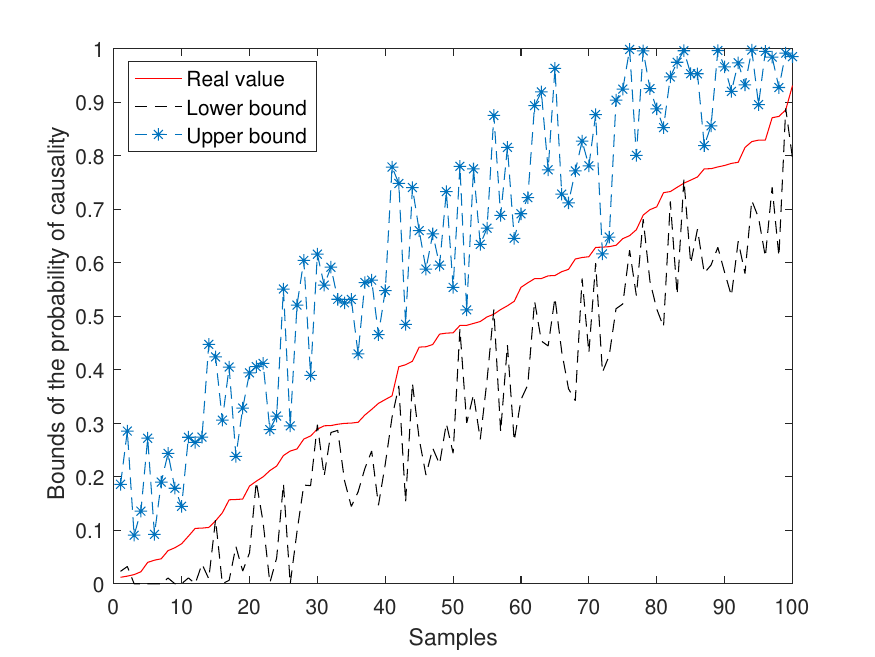}    
  \caption{Bounds of the $P{{C}_{A}}(T=2, t=1)$ in complete mediation analyses for 100 samples}\label{fig2}
\end{figure}

Figure 2 shows that, in general, the estimates of the upper and lower bounds on the probability of causality are closely around the real values.
Similarly to the simulation above, we again randomly generated 100 samples of $P{{C}_{A}}$ and sample distributions campatible with $P{{C}_{A}}$. We then draw the graph of the real value of the probability of causality, the midpoints of the bounds proposed through Theorem 1, and the midpoints of the bounds proposed through Theorem 2. The results are shown in Figure 3, where the corresponding random number generation mechanism is consistent with Figure 2.

\begin{figure}[H]                           
  \centering                               
  \includegraphics[scale=0.7]{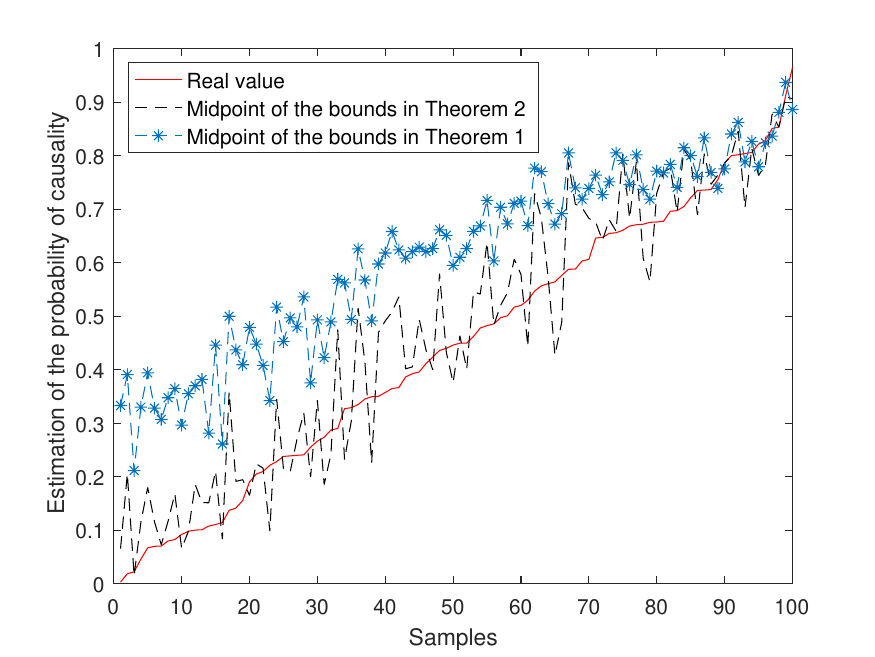}     
  \caption{The midpoints of the upper and lower bounds of the $P{{C}_{A}}(T=2, t=1)$ proposed in Theorem 1 and Theorem 2 for 100 samples }\label{fig3}
\end{figure}

From Figure 3, the midpoints of the bounds on the probability of causality proposed through Theorem 2 are good estimates of real values. In contrast, the midpoints of the bounds proposed by Theorem 1 provide better estimates when the real values are close to 1, but show large deviations when the real values are close to 0. The bounds proposed in Theorem 1 and Theorem 2 have an average gap (Upper bound - Lower bound) of 0.58 and 0.28 across 100 samples, respectively. This further suggests that it is more accurate to use the bounds proposed in Theorem 2 in complete mediation analysis.

\section{Discussion}
This article makes statistical inferences about the probability of causality in a specific problem situation for binary and ordinal outcomes in turn. The main contribution of this article is to provide bounds on the probability of causality under ordinal outcomes and to illustrate improved upper and lower bounds that can be obtained by taking into account additional information about the mediator variable in cases where exposure has no direct effect on the outcome. Although examples of complete mediation in practical research are uncommon, this paper still has theoretical significance and academic research value. In addition, in some problems of interest, there may be covariates that affect both the exposure and outcome variables, in which case Assumptions 1 and 3 can be varied to hold after adjusting for the covariates, and the study of such problems under ordinal outcomes is left for future work.

\section*{Acknowledgments}
The research is partially supported by the National Natural Science Foundation of China (Grant no. 12001334), the Natural Science Foundation of Shandong Province (Grant no.  ZR2020QA022, ZR2021MA048).

\section*{Disclosure statement}
The authors declare that they have no conflict of interest.

\end{document}